\documentclass[12pt,reqno]{amsart}

\usepackage[koi8-r,cp1251]{inputenc}
\usepackage{amsfonts,amssymb,amsmath, hyperref}
\usepackage[final]{epsfig}
\usepackage{graphicx}

\usepackage[mathlines]{lineno}

\usepackage{times}
\usepackage{amsfonts}
\usepackage{amssymb}
\usepackage{latexsym}
\usepackage{xcolor}

\usepackage{comment}

\newtheorem{theorem}{Theorem}
\newtheorem{lemma}[theorem]{Lemma}
\newtheorem{prop}[theorem]{Proposition}
\newtheorem{definition}[theorem]{Definition}
\newtheorem{remark}{Remark}
\newtheorem{corollary}[theorem]{Corollary}

\begin{document}

\begin{center}
\Large{On Toeplitz operators on compact Abelian groups  and discrete Wiener--Hopf operators}
 \end{center}
 
 \ 
 
\centerline{A. R. Mirotin}

\centerline{amirotin@yandex.ru}

\

Highlights

\begin{itemize}

\item
The rotation number can be defined for any continuous non-degenerate two-dimensional vector field on a compact connected Abelian group.
\item
A  Gohberg–Krein index formula is valid  for semi-Fredholm Toeplitz operators with continuous symbols defined on compact connected Abelian groups

\item
Criteria  are valid  for these operators to be semi-Fredholm in terms of their symbol

\item
Essential spectra can be described for these operators  in terms of their symbol

\item
Similar results are valid for discrete Wiener–Hopf operators

\end{itemize}

ABSTRACT
 This paper introduces the concept of a rotation number for a continuous, non-degenerate two-dimensional vector field (a zero-free complex-valued function) on a compact connected Abelian group. This concept generalizes the notion of a finite rotation number for such groups, previously introduced by the author. Using this concept, a Gohberg-Krein index formula is derived for semi-Fredholm Toeplitz operators with continuous symbols defined on such groups. Criteria for these operators to be semi-Fredholm are established, and their essential spectra are described.
As a by-product for the continuous symbol case, conditions for Fredholmness and semi-Fredholmness are established, and the Fredholm index of Wiener-Hopf operators over a linearly ordered discrete Abelian group is calculated in terms of their symbols. Spectral properties-including the spectra and essential spectra—of the Wiener–Hopf operators under consideration are also described.

\


\section{INTRODUCTION AND PRELIMINARY INFORMATION}

This paper is devoted to Toeplitz operators on compact and connected Abelian groups and Wiener--Hopf operators on discrete ordered Abelian groups, which are closely related.

For the theory of Toeplitz operators over the circle group see the classical  monograph \cite{BS}.  The study of  Toeplitz operators  on compact and connected Abelian groups $G$ was initiated by G.~J.~ Murphy \cite{Murphy} for Hurdy spaces $H^2(G)$. Later, in the paper \cite{SbMath}, the author developed the Fredholm and spectral properties of  Toeplitz operators in spaces $H^p(G)$ ($1<p<\infty$).

 The one-dimensional discrete Wiener-Hopf operators  were studied  by M.~ G.~ Krein using the factorization method. See \cite{BS} and the bibliography therein for further investigations. Such operators in various function spaces continued to be studied and are still being studied today by 
 Bottcher ~A.,  Grudsky~ S.~M., Gohberg~ I., Goldberg S., Kaashoek~ M.~A,   Karlovych ~O., Nowak~  M.~A.,  Duduchava~ R.~V., Silbermann B., Thampi~ S.~M. and many other researches.

The factorization method was used by L.~ Goldenstein and I.~ Gokhberg to study the Wiener-Hopf equation on the group $\mathbb{Z}^n$.   

Wiener--Hopf operators for the case of invertible and summable kernels in several spaces over general  Abelian groups were considered by V.~ Adukov in \cite{Adukov}. In this paper  he introduced the concept of a Wiener--Hopf operator over a discrete ordered Abelian group and discussed in detail the invertibility and Wiener--Hopf factorization of these operators in terms of their symbols. However, no criterion of Fredholmness and index formula were obtained, and  spectra of these operators were not studied. The paper
\cite{Adukov} also provides a detailed history of the topic and a bibliography of the first stage of the study of these operators. These operators continued to be studied subsequently (see, e.~g.,  \cite{EMRS} and the bibliography therein).

As noted in \cite{SbMath}, a characteristic feature of compact Abelian groups not isomorphic to the circle group is the existence of semi-Fredholm, but not Fredholm, Toeplitz operators with continuous symbols.
The goal of the present research was to obtain  the Gohberg--Krein index formula for the case of semi-Fredholm  Toeplitz operators on compact and connected Abelian groups and  to study  their essential spectra. As a by-product we  obtain  an explicit formula for the Fredholm index of Wiener--Hopf operators over a linearly ordered discrete Abelian groups and  spectral properties     of these operators  in terms of their symbol. The above mentioned problems were solved using the introduced and studied generalization of the concept of the rotation number of a vector field on a compact connected Abelian group, previously defined by the author for the case of finite rotation in \cite{SbMath}. 

Some of the main results of the present paper were announced in \cite{arxiv}.

Everywhere below, $X$ denotes a discrete
linearly ordered Abelian group,
$X_+:=\{\chi\in X: 1\le \chi\}$ is the positive cone in $X$. Then $X_+$ is  a  subsemigroup of $X$ containing
the identity character $1$,
where $X_+\cap X_+^{-1}=\{1\}$ and $X=X_+\cup X_+^{-1}$. Conversely if a  subsemigroup $X_+$ of $X$ enjoys  the above mentioned properties  then
 it induces a linear ordering on $X$ compatible
with the group structure, according to the rule $\xi\leq\chi:=\chi\xi^{-1}\in
X_+$. In particular for $\xi, \chi\in X_+$  the inequality $\xi < \chi$ means that $[1,\xi)\subset [1,\chi)$ where the order interval $[1,\chi)=X_+\setminus \chi X_+$. Similarly  for $\xi, \chi\in X_-:=X\setminus X_+$  the inequality $\xi < \chi$ means that $(\chi,1]\subset (\xi,1]$. 

Below we put also $X_{-}:=X\setminus  X_+$.

As is well known, a (discrete) Abelian group $X$ can be linearly ordered if and only if it is torsion-free
(see, e.g., \cite{Rud}), which, in turn, is equivalent to the fact that
its character group $G$ is compact and connected \cite{Pont}. However, the linear ordering
on $X$ is, in general, not unique. In what follows, we will assume that this ordering is chosen and fixed. Main results below depend of the order we choose  heavily.

We denote by $G:=\widehat{X}$ the dual (the group of characters) of $X$. Then $G$ is an Abelian compact group and $X=\widehat{G}$ by the Pontryagin-van Kampen theorem.

In applications, the role of $X$ is often played by subgroups of the additive group $\mathbb{R}^n$,
endowed with the discrete topology. In this situation, $G$ is a Bohr compactification of $X$ (see, e.g., the paper \cite{EMRS} and references therein).

By $\widehat{\varphi}$ (or $\mathcal{F}\varphi$) we denote the Fourier transform of a function $\varphi$ in $L^1(G)$ or $L^2(G)$, i.~e. (with the corresponding interpretation of the integral)
$$
\widehat{\varphi}(\xi)=\int_G\varphi(x)\overline{\xi(x)}\,dx,\ \ \xi\in X,
$$
where $dx$ is the normalized Haar measure of $G$.

By $\check k$ we denote the inverse Fourier transform $\mathcal{F}^{-1}k$ of a function $k\in \ell^2(X_+)$, i~.e.,
$$
\check k(x)=\sum_{\chi\in X}k(\chi)\chi(x), \ \ x\in G.
$$
where the space
$$
\ell^2(X_+):=\{f:X_+\rightarrow\mathbb{C}:\sum\limits_{\chi\in X_+}| f(\chi)| ^2<\infty \}$$
(here each function $f$ has at most countable support) is endowed with
the scalar product
$$
\langle f,g\rangle=\sum\limits_{\chi\in X_+}f(\chi)\overline{g(\chi)}.
$$
The spaces $\ell^p(X_+)$ are defined similarly. Obviously, the system $\{1_{\{\chi\}}\}_{\chi\in
X_+}$ (we denote by $1_A$ the indicator of the set $A$) of indicators of one-point subsets is an orthonormal basis of the space $\ell^2(X_+)$.

In what follows, we assume that $\ell^p(X_+)$ is naturally embedded in $\ell^p(X)$.

\section{The Rotation Index of a Two-Dimensional Vector Field on a Compact Abelian Group}

In \cite{SbMath}, \S 3, the rotation index of two-dimensional vector fields on a group $G$ was defined and its important properties were discussed. This index is a generalization of the classical rotation index on the circle (see, e.g., \cite{Roe}). In the course of the present research, this concept was generalized to the infinite case and its main properties, formulated below, were established.

Furthermore, $C(G)^{-1}$ will denote the group of invertible elements of the algebra $C(G)$ of continuous complex-valued functions on a compact Abelian group $G$; the symbol $\#$ denotes the number of elements in a finite set, or $+\infty$ if the set is infinite. 

We define the rotation index of functions from $C(G)^{-1}$ in two steps.

\begin{definition}\label{ Definition 1} \textit{The rotation index  of a character $\chi$} of a group $G$ is defined as follows:
$$
\mathrm{ind}\chi=
\begin{cases}\#[1,\chi ), & \mbox{ if }\chi\in X_{+}\\
-\mathrm{ind}\,\chi^{-1}, & \mbox{ if }\chi\in X_{-}.
\end{cases}$$
\end{definition}

Let $X^{i}$ be a subgroup of characters with finite indices. According to \cite[Theorem 2]{SbMath}, $X^{i}=\{\chi_{1}^n,n\in\mathbb{Z}\}$ if $X^{i}$ contains the first strictly  positive element $\chi_{1},$ and $X^{i}=\{1\}$ otherwise.

The following properties of the character index are established.

\begin{theorem}\label{Theorem 1}
(i) The mapping $\mathrm{ind}:\,X\rightarrow\mathbb{Z}\cup\{\pm\infty\}$ preserves the (linear)  order in the group $X$. More precisely, the following statement holds: if $\chi,\chi'\in X$, then
$$
\mathrm{ ind}\,\chi<\mathrm{ind}\,\chi'\Rightarrow\chi<\chi';
$$
(ii) If the group $X$ contains the smallest strictly positive element $\chi_{1}$, then for all $\chi\in X^{i}$ the equality
$$
\chi=\chi_{1}^{{\rm ind}\,\chi} 
$$
holds.
\end{theorem}

\begin{proof} (i) Let $\mathrm{ ind}\,\chi<\mathrm{ind}\chi'$. The case where $\chi, \chi'\in X^i$ was considered in \cite{SbMath}. Otherwise  the following cases are possible:

1) ${\rm ind}\,\chi\not=\infty, {\rm ind}\,\chi'=+\infty$.  One can assume that $\chi\in X_+\setminus \{1\}$.  Then $\#[1,\chi)<\infty$, but $\#[1,\chi')=+\infty$. Therefore $[1,\chi) \subset  [1,\chi')$, i.~e. $\chi<\chi'$.

2) ${\rm ind}\,\chi=-\infty, {\rm ind}\,\chi'\not=\infty$. Then ${\rm ind}\,\chi^{-1}=+\infty, {\rm ind}\,(\chi')^{-1}\not=\infty$. Then  $(\chi')^{-1}<\chi^{-1}$by 1)  and therefore $\chi<\chi'$.

3) ${\rm ind}\,\chi=-\infty, {\rm ind}\,\chi'=+\infty$. In this case $\chi\in X_-$, $\chi'\in X_+$.

(ii) This follows from the fact that if the group $X$ contains the smallest strictly positive element $\chi_{1}$, then the map ${\rm ind}:X^i\to\mathbb{Z}$ is a group and order isomorphism \cite[Theorem 2]{SbMath}. Indeed, it follows that ${\rm ind}\,\chi_1=1$, and $X^i=\{\chi_1^n: n\in \mathbb {Z}\}$. Then ${\rm ind}(\chi_1^n)=n$, since   ${\rm ind}$ is a group homomorphism  (the logarithmic property of index \cite[Theorem 2]{SbMath}).
\end{proof}

Below, we use the facts that any function $\varphi\in C(G)^{-1}$ can be represented in the form 
\begin{equation}\label{BvK}
\varphi=\chi_\varphi e^{g},g\in C(G),\chi_\varphi\in X
\end{equation}
(the Bohr--van Kampen decomposition \cite{van Kampen}), and the character
$\chi_\varphi$ in this decomposition is uniquely determined by $\varphi$  (see \cite[Section 2 of \S 2]{Gorin} and \cite[Theorem 2.1]{Lin}   for the wide generalizations of   these facts).

\begin{definition}\label{Definition2}The rotation index of a complex function  $\varphi\in C(G)^{-1}$ (two-dimensional non-singular real vector field) is the rotation index of its character $\chi_\varphi$ in the Bohr-van Kampen decomposition, i.~e.,
$$
{\rm ind}\varphi:={\rm ind}\chi_\varphi \mbox{ if } \varphi\in C(G)^{-1}.
$$
\end{definition}

Recall that $\chi_1$ denotes the first strictly positive element in $X_+$ (if such element exists).

\begin{theorem}\label{ind_phi} Let  $\varphi\in C(G)^{-1}$. If  $\chi_\varphi=1$ then $\mathrm{ind}\varphi=0$. If  $\chi_\varphi\ne 1$ one has
$$\mathrm{ind}\varphi
=\begin{cases}
n, &\mbox{ if } \exists \chi_1 \mbox{ and } \chi_\varphi=\chi_{1}^{n}\ (n\in\mathbb{Z})\\
\pm\infty, & \mbox{ otherwise and } \chi_\varphi \in X_{\pm}.
\end{cases}
$$
\end{theorem}

\begin{proof} 1)The first statement is obvious.

2)  If    $\chi_1$ exists,  and  $\chi_\varphi=\chi_{1}^{n}$ with $n\in\mathbb{Z}$ then
$\mathrm{ind}\chi_\varphi=n$ as was shown above in the proof of Theorem \ref{Theorem 1}.

3) Let $\chi_1$ exists. Recall that $X^i=\{\chi_{1}^{n}: n\in\mathbb{Z}\}$. Thus, if $\chi_\varphi\not \in \{\chi_{1}^{n}: n\in\mathbb{Z}\}$, then 
$\mathrm{ind}\chi_\varphi=\infty$, and, by definition, $\mathrm{ind}\chi_\varphi\ge 0 \ (\le 0)$ for
$\chi_\varphi \in X_{\pm}.$

4) Let $\chi_1$  not exist. Let   $\chi_\varphi \in X_{+}\setminus \{1\}$. Since $\chi_1$ not exists, one has ${\rm ind}\chi=\#[1,\chi)=+\infty$.  Finally, if $\chi_\varphi \in X_{-}$ then  $\chi_\varphi^{-1} \in X_{+}$ and the result follows.

\end{proof}

The following proposition is a special case of Theorem \ref{ind_phi}.

\begin{prop}\label{Proposition 1} 1) Let $X$ contain the first strictly positive element $\chi_1$.

Then Definition \ref{ Definition 1} is equivalent to the following:
$$\mathrm{ind}\chi
=\begin{cases}
n, & \mbox{ if } \chi=\chi_{1}^{n}\ (n\in\mathbb{Z})\\
+\infty, & \mbox{ if } \chi\in X_{+}\setminus X^{i}\\
-\infty, & \mbox{ if } \chi\in X_{-}\setminus X^{i}.
\end{cases}
$$
2) If $X$ has no first positive element, then Definition \ref{ Definition 1} is equivalent to the following:
$$\mathrm{ind}\chi
=\begin{cases}
0, , & \mbox{ if } \chi=1\\
+\infty, & \mbox{ if } \chi\in X_{+}\setminus \{1\}\\
-\infty, & \mbox{ if } \chi\in X_{-}.
\end{cases}
$$
\end{prop}

Following \cite{SbMath} we introduce the notation 
$$\Phi(G):=\{\varphi\in C(G)^{-1}: \mathrm{ind}\varphi\ne\infty\}.$$ 
Then $\Phi(G)=X^ie^{C(G)}$ is an Abelian group with respect to pointwise multiplication of functions. It is known \cite{SbMath} that $\Phi(G)=C(G)^{-1}$ if and only if $G$ is isomorphic to the one-dimensional torus $\mathbb{T}$.

The index of a two-dimensional vector field has the following properties.

\begin{theorem}\label{Theorem 2}
(i) (Logarithmic property.) The index of the product of two functions in $C(G)^{-1}$ is equal to the sum of their indices, if this sum is meaningful.

(ii) (Homotopy invariance.) If two functions in $C(G)^{-1}$ are homotopic, then their rotation indices are equal.

(iii) Conversely, if two functions in $C(G)^{-1}$ have the same, and finite, rotation index, then these functions are homotopic.
\end{theorem}

\begin{proof} One can assume that $X^i$ is not trivial. For the case of a finite index properties (i) and (ii) follows from \cite[Theorem 2, Remark 2]{SbMath}. Otherwise  in view   of the definition of the index it is sufficient to prove the property (i)  for characters only. Several cases are possible.

1) $\mathrm{ind}\chi=+\infty, \mathrm{ind}\chi'\not=\infty$. Then $\chi>\xi$ for all $\xi\in X^i$ and therefore $\chi\chi'>\xi\chi'$ for all $\xi\in X^i$. This implies by monotony of the index  that 
$$
\mathrm{ind}(\chi\chi')>\mathrm{ind}(\xi\chi')=\mathrm{ind}\xi+\mathrm{ind}\chi'
$$ 
for all $\xi\in X^i$. It follows that $\mathrm{ind}(\chi\chi')=+\infty=\mathrm{ind}\chi+\mathrm{ind}\chi'$.

2) $\mathrm{ind}\chi=-\infty, \mathrm{ind}\chi'\not=\infty$. Then $\mathrm{ind}(\chi^{-1})=+\infty, \mathrm{ind}(\chi'^{-1})\not=\infty$ and thus
$$
\mathrm{ind}((\chi\chi')^{-1})=\mathrm{ind}(\chi^{-1})+\mathrm{ind}(\chi'^{-1}).
$$
The result follows, since $\mathrm{ind}(\chi^{-1})=-\mathrm{ind}\chi$ by definition.

3) $\mathrm{ind}\chi= \mathrm{ind}\chi'=+\infty$. By Proposition  \ref{Proposition 1} one has 
 $\chi, \chi'\in X_+$ and so  $\chi\le \chi\chi'$. It follows by monotony of the index  that 
 $\mathrm{ind}(\chi\chi')\ge \mathrm{ind}\chi =+\infty$.
 
 4) $\mathrm{ind}\chi= \mathrm{ind}\chi'=-\infty$. The proof in this case is similar to the proof of the case 3).

(ii) For functions from $\Phi(G)$ this fact was mentioned in \cite[Remark 2]{SbMath}. Now let functions $\varphi_1,\varphi_2\in C(G)^{-1}$ are homotopic. As is well known,
the set $\exp(C(G))$ is open and is the identity component of $C(G)^{-1}$. It follows that $\varphi_1,\varphi_2$ belong to  $\varphi\exp(C(G))$ the connected component of some function 
$\varphi\in C(G)^{-1}$. Due to the uniqueness of the character factor in the Bohr-van Kampen decomposition we have $\chi_{\varphi_1}=\chi_{\varphi_2}=\chi_{\varphi}$, and thus $\mathrm{ind}\varphi_1=\mathrm{ind}\varphi_2=\mathrm{ind}\varphi$.

(iii) If  $\mathrm{ind}\varphi_1=\mathrm{ind}\varphi_2$  for  $\varphi_1,\varphi_2\in \Phi(G)$   then $\chi_{\varphi_1}=\chi_{\varphi_2}=\chi\in X^i$, due to the monotony of $\mathrm{ind}$. It follows that $\varphi_j=\chi e^{g_j}$ for some  $g_j\in C(G)$ ($j=1,2$). Then the function
$$
F(t, x)=\chi(x)e^{t g_2(x)+(1-t)g_1(x)} \quad  (t\in [0,1], x\in G)
$$
realizes the homotopy in  $C(G)^{-1}$  between  $\varphi_1$ and $\varphi_2$. 
\end{proof}

\begin{remark} For the case of infinite index the statement (iii) of Theorem \ref{Theorem 2} is not correct. Indeed, if $X$ has no  first strictly positive element, and $\chi, \chi' \in X_+\setminus\{1\}$,  then  $\mathrm{ind}\chi= \mathrm{ind}\chi'=+\infty$ by  Proposition \ref{Proposition 1}. On the other hand, if  $\chi\not= \chi'$ then $\chi$ and  $\chi'$ are not homotopic by \cite[Lemma 2.2]{Lin}.
\end{remark}

\begin{theorem}\label{uniqueness} The rotation index defined above is the only non-trivial homomorphism of groups $\Phi(G)\to \mathbb{Z}$, up to an integer factor, that has the property of homotopy invariance.
\end{theorem}

\begin{proof} Let $h:\Phi(G)\to \mathbb{Z}$ be a  group homomorphism that satisfies the conditions of the theorem. Then $h(1)=0$.Two cases are possible.

1)  $X_+$ contains  the first strictly positive element $\chi_1$. Since a function $\varphi=\chi_\varphi e^{g}\in C(G)^{-1}$ is homotopic  to $\chi_\varphi$  ($F(t,x)=\chi_\varphi(x) e^{tg(x)}, \ t\in [0,1]$ is a homotopy) on has 
$$
h(\varphi)=h(\chi_\varphi)=h(\chi_1^{\mathrm{ind}(\chi_\varphi)})=m \mathrm{ind}(\chi_\varphi)=m \mathrm{ind}(\varphi),
$$
where $m:=h(\chi_1)$ is a nonzero integer, since $h$ is nontrivial. 

2)  $X_+$ does not contain  the first strictly positive element. Then $X^i=\{1\}$ by \cite[Theorem 2, Remark 1]{SbMath}. It follows that $\chi_\varphi=1$ for every function $\varphi\in C(G)^{-1}$. Thus   $\varphi$ is homotopic  to $1$, and the result follows.

\end{proof}

\section{Fredholmness and the index of generalized Toeplitz operators}

Recall that for a bounded operator $T$ on a Banach space $Y$, the notation $T\in\Phi_+(Y)$ means that its image ${\rm Im}T$ is closed and the kernel
${\rm Ker}T$ is finite-dimensional, while the notation $T\in\Phi_-(Y)$ means that the quotient space
$Y/{\rm Im T}$ is finite-dimensional; operators in $\Phi_\pm(Y):=\Phi_-(Y)\cup\Phi_+(Y)$ are called semi-Fredholm,
and operators in $\Phi(Y):=\Phi_-(Y)\cap\Phi_+(Y)$ are called Fredholm in $Y$. In both cases,
$$
{\rm Ind}T:={\rm dim}{\rm Ker}T-{\rm dim}{\rm Coker}T
$$
where ${\rm Coker}T:=Y/{\rm Im}T$.

It will be convenient to generalize the definition of a Toeplitz operator over a group $G$ \cite{Murphy} (see also \cite{SbMath}) to the case of a symbol $\varphi\in L^2(G)$.

\begin{definition}\label{Definition  3} Let $1\le p\le\infty$. The Hardy space $H^p(G)$ over a group
$G$ is defined as follows:
$$
H^p(G)=\{f\in L^p(G):\widehat{f}(\chi)=0\ \forall\chi\notin X_+\}.
$$

We endow $H^p(G)$ with $L^p$ norm.
\end{definition}
\bigskip

\begin{definition}\label{Definition4}  Let $\varphi\in L^2(G)$. We define the Toeplitz operator $T_\varphi$ in $H^2(G)$ initially on the set of analytic polynomials ${\rm Pol}_+(G)$ (the linear span of $X_+$) by the equality
$$
T_\varphi q=P_+(\varphi q) \ (q\in {\rm Pol}_+(G)),
$$
where $P_+:L^2(G)\to H^2(G)$ is the orthogonal projection.
\end{definition}

The following theorem is the Theorem 4 from \cite{SbMath}.

\begin{theorem} (Fredholmness and the Index formula.)\label{th:Fredholm} Let $\varphi\in C(G)$. The Toeplitz operator $T_{\varphi}$ on $H^2(G)$ is  Fredholm if and only if $\varphi\in \Phi(G)$. In this case,
\begin{equation}\label{Ind T}
{\rm ind}T_\varphi=-{\rm ind}\varphi.
\end{equation}
\end{theorem}

\begin{corollary}\label{corollary 5.14Toeplitz}
1)  Let   $\varphi_{1}, \varphi_{2}\in C(G)^{-1}$. If symbols  $\varphi_{1}$ and $\varphi_{2}$ are homotopic then ${\rm Ind}T_{\varphi_{1}}={\rm Ind}T_{\varphi_{2}}$.

2)   Let  $\varphi_{1}, \varphi_{2}\in \Phi(G)$. If  ${\rm Ind}T_{\varphi_{1}}={\rm Ind}T_{\varphi_{2}}$  then symbols $\varphi_{1}$ and $\varphi_{2}$ are homotopic.
\end{corollary}
This follows from  Theorems \ref{th:Fredholm} and \ref{Theorem 2}.

\section{Semi-Fredholmness and index of generalized Toeplitz operators}

The following lemma is well known. 

\begin{lemma}\label{L_R_inverse} Let $X,Y$ be Banach spaces and $T:X\to Y$ a bounded operator.

(i) Operator $T$ has a bounded left inverse if and only if it is injective and its image ${\rm Im}T$ is closed.

(ii) Operator $T$ has a bounded right inverse if and only if if it is surjective and its kernel ${\rm Ker}T$ is a complemented subspace of $Y$. 
\end{lemma}

It follows from \cite[Corollary 6]{SbMath},   that 
 there exists an operator $T_\varphi$ with $\varphi \in C(G)^{-1}$ that is not Fredholm  if (and only if)  $ G\ne \mathbb{T}$. 
 But the following theorem holds true.

\begin{theorem}\label{th:Semi-Fredholm}   Let $\varphi\in C(G)$. 

(i) The following statements are equivalent:

(a) $T_\varphi$ is a semi-Fredholm operator; 

(b) $\varphi\in C(G)^{-1}$; 

(c) the operator $T_\varphi$ has  left or right  bounded inverse. 

If (a) --- (c) hold  $T_\varphi$ has a  left (right)  bounded inverse if and only if $\chi_\varphi\in X_-\cup \{1\}$ (respectively $\chi_\varphi\in X_+$).

(ii)  For all  $\varphi\in C(G)^{-1}$ the equality \ref{Ind T} is valid.
\end{theorem}

\begin{proof} (i) 
$(a)\Longrightarrow(b).$ This is Theorem 3 in \cite{SbMath}.

$(b)\Longrightarrow(c).$ In this case, $\varphi=\chi_\varphi e^{g}, g\in C(G), \chi_\varphi\in X$.  Two cases are possible. 

(i.1). $\chi_\varphi\in X_-\cup\{1\}$. Then $T_{\chi_\varphi}$ is surjective  because $\overline{\chi_\varphi}\in X_+$ and therefore $T_{\chi_\varphi}(\overline{\chi_\varphi} f)=P_+f =f$ for all $f\in H^2(G).$ Moreover $T_\varphi=T_{\chi_\varphi}T_{e^g}$ by \cite[Proposition 1]{SbMath}. Since $T_{e^g}$ is bijective,   $T_\varphi$ is surjective. Finally, ${\rm Ker} T_{\chi_\varphi}$ is closed, and $T_\varphi$ has a bounded right inverse   by Lemma \ref{L_R_inverse} (ii). In particular  $T_\varphi\in\Phi_-(H^2(G))$.

(i.2). $\chi_\varphi\in X_+$.  In this case, $T_\varphi=T_{e^g}T_{\chi_\varphi}$ by \cite[Proposition 1]{SbMath}  and this operator is injective, because  $T_{\chi_\varphi}$ is injective (recall that $T_{e^g}$ is bijective).  Indeed, let $T_{\chi_\varphi}f=0$ for some $f\in H^2(G)$. Since $\chi_\varphi f\in H^2(G)$, it follows $\chi_\varphi f=0$ and thus $f=0$.

Further, ${\rm Im}T_\varphi=T_{e^g}{\rm Im}T_{\chi_\varphi}$ is closed in $H^2(G)$, since ${\rm Im}T_{\chi_\varphi}=\chi_\varphi H^2(G)$ is closed in $H^2(G)$. So, $T_\varphi$ has a bounded left  inverse   by Lemma \ref{L_R_inverse} (i). In particular $T_\varphi\in\Phi_+(H^2(G))$.

$(c)\Longrightarrow(a).$ If $T_\varphi$ has a bounded left  inverse then by   Lemma \ref{L_R_inverse} (i) one has ${\rm Ker}T_\varphi=\{0\} $ and ${\rm Im}T_{\varphi}$ is closed in $H^2(G)$. This  implies $T_{\varphi}\in \Phi_+(H^2(G))$.

On the other hand, if $T_\varphi$ has a bounded right  inverse then by   Lemma \ref{L_R_inverse} (ii) one has ${\rm Im}T_{\varphi}=H^2(G)$, and therefore  $T_{\varphi}\in \Phi_-(H^2(G))$.

(ii) The case $\varphi\in \Phi(G)$, i.~e., $\chi_\varphi\in X^i$ was considered in \cite[Theorem 4]{SbMath}. In this case $T_{\varphi}\in \Phi(H^2(G))$ and \ref{Ind T} is valid (see Theorem \ref{th:Fredholm} above).  Now let  $\chi_\varphi\not\in X^i$. Since $T_\varphi$ is a semi-Fredholm operator by (i), it remains to consider the following two cases.

(ii.1)  $T_{\varphi}\in \Phi_+(H^2(G))\setminus \Phi(H^2(G))$. Since the kernel
${\rm Ker}T_{\varphi}$ is finite-dimensional and ${\rm Ind}T_{\varphi}=\infty$, we have
$$
{\rm Ind}T_{\varphi}={\rm dim}{\rm Ker}T_{\varphi}-{\rm dim}{\rm Coker}T_{\varphi}=-\infty.
$$
It follows that $\chi_\varphi\in X_+$ because for $\chi_\varphi\in X_-$ we saw in (i.1) that $T_\varphi$ is surjective and therefore $T_{\varphi}\in \Phi_-(H^2(G))$. Since $T_{\varphi}\not\in\Phi(H^2(G))$, $\varphi\not\in \Phi(G)$ and so $\chi_\varphi\not\in X^i$. It follows that ${\rm ind}\varphi={\rm ind}\chi_\varphi=+\infty$ and \ref{Ind T} is valid.

(ii.2) If $T_{\varphi}\in \Phi_-(H^2(G))\setminus \Phi(H^2(G))$ then $\chi_\varphi\in X_-$ because for $\chi_\varphi\in X_+$ we saw in (i.2) that $T_{\varphi}$ is injective and ${\rm Im}T_{\varphi}$ is closed in $H^2(G)$ and thus $T_{\varphi}\in \Phi_+(H^2(G))$. Therefore, as above,  ${\rm ind}\varphi={\rm ind}\chi_\varphi=-\infty$.

On the other hand, since  $T_\varphi$ is surjective (see part (ii.1)),  ${\rm dim}{\rm Coker}T_{\varphi}=0$ and therefore ${\rm Ind}T_{\varphi}=+\infty$. 

The last statement in (i) follows  from the previous considerations.

This completes the proof.
\end{proof}

The proof of Theorem \ref{th:Semi-Fredholm}  implies the following

\begin{corollary}\label{cr:semiFredholm} 
Let  (a) --- (c) hold. The operator  $T_\varphi$ belongs to $ \Phi_+(H^2(G))$ ($\Phi_-(H^2(G))$) if and only if $\chi_\varphi\in X_+\cup X^i$ (respectively $\chi_\varphi\in X_-\cup X^i$).
\end{corollary}

\begin{proof}
Indeed, if $T_\varphi\in\Phi(H^2(G))$ then $\chi_\varphi\in X^i$. And if $T_{\varphi}\in \Phi_+(H^2(G))\setminus \Phi(H^2(G))$ then as was shown in the part (ii.1) of the proof of Theorem \ref{th:Semi-Fredholm} $\chi_\varphi\in X_+$.

Conversely, if $\chi_\varphi\in X^i$ then  $T_\varphi\in\Phi(H^2(G))$. And if $\chi_\varphi\in X_+$ then as was shown in the part (i.2) of the proof of Theorem \ref{th:Semi-Fredholm} $T_{\varphi}\in \Phi_+(H^2(G))$.

The proof of the second statement of the corollary is similar to the previous one.
\end{proof}

Corollary \ref{cr:semiFredholm} gives us an opportunity to compute the semi-Fredholm spectra of our Toeplitz operators. Recall the following notions.

Let $T$ be a bounded operator in a complex Banach space $Y$. The set
$$
\sigma_{le}(T):=\{\lambda\in\mathbb{C}: T-\lambda\not\in \Phi_+(Y)\}
$$
is called a {\it  left semi-Fredholm spectrum} of $T$.

Similarly,  the set
$$
\sigma_{re}(T):=\{\lambda\in\mathbb{C}: T-\lambda\not\in \Phi_-(Y)\}
$$
is called a {\it  right semi-Fredholm spectrum} of $T$.

By {\it holes} of a compact connected set
$K\subset \mathbb{C}$, we, as usual, understand the bounded components of its complement $\mathbb{C}\setminus K$.

\begin{theorem}\label{th:SpectrasemiFredholm} 
Let $\varphi \in C(G)$. Then the following statements hold:

1) the  left semi-Fredholm spectrum $\sigma_{le}(T_\varphi)$ is obtained from the set $\varphi(G)$ by
joining the holes $\Lambda$ (if they exist) such that $\chi_{\varphi-\lambda}\in X_-\setminus X^i$ for some (for all) $\lambda\in \Lambda$;

2) the  right  semi-Fredholm spectrum $\sigma_{re}(T_\varphi)$ is obtained from the set $\varphi(G)$ by
joining the holes $\Lambda$ (if they exist) such that $\chi_{\varphi-\lambda}\in X_+\setminus X^i$ for some (for all) $\lambda\in \Lambda$.

3) Spectra  $\sigma_{le}(T_\varphi)$ and  $\sigma_{re}(T_\varphi)$ are connected.
\end{theorem}

\begin{proof}  1) The Bohr-van Kampen theorem and the uniqueness of $\chi_\varphi$ means that the factor group $C(G)^{-1}/\exp(C(G))$ can be identified with $X$. We denote by $\pi$ be the canonical homomorphism
$$
C(G)^{-1}\to C(G)^{-1}/\exp(C(G))=X, \varphi\mapsto \chi_\varphi.
$$
For a fixed $\varphi \in C(G)$ consider the mapping
$$
\alpha_\varphi: \mathbb{C}\setminus\varphi(G)\to
X, \lambda\mapsto\pi(\varphi-\lambda)=\chi_{\varphi-\lambda}.
$$
Since $X$ is discrete and $\alpha_\varphi$ is continuous, this map is constant on every hole of the (open) set $\mathbb{C}\setminus\varphi(G)$. 

On the other hand, by Corollary \ref{cr:semiFredholm}
\begin{align}\label{sigma_le}
\sigma_{le}(T_\varphi)&=\{\lambda\in\mathbb{C}: T_{\varphi-\lambda}\not\in \Phi_+(H^2(G))\}\\   \nonumber
&=\varphi(G)\cup\{\lambda\in\mathbb{C}\setminus\varphi(G):\chi_{\varphi-\lambda}\not\in X_+\cup X^i\}\\ \nonumber
&=\varphi(G)\cup\{\lambda\in\mathbb{C}\setminus\varphi(G):  \chi_{\varphi-\lambda}=\alpha_{\varphi}(\lambda)\in X_-\setminus X^i\}.\\ \nonumber
\end{align}
 But for $\lambda\in\mathbb{C}\setminus\varphi(G)$ the function $\varphi-\lambda$ is not invertible  in $L^\infty(G)$ and therefore $T_{\varphi-\lambda}$ is not semi-Fredholm in $H^2(G)$ by \cite[Theorem 3]{SbMath}. Moreover, 
since the operator $ T_{\varphi-\lambda}$ is invertible for large values of $\lambda$ and since $\lambda\mapsto\chi_{\varphi-\lambda}$ is constant
on connected components of the set $\mathbb{C}\setminus \varphi(G)$, it follows that the unbounded
component of this set is disjoint from $\sigma_{le}(T_\varphi)$. Thus formula \eqref{sigma_le} is equivalent to the statement of the part 1) of the theorem.

2) The proof of this part is  similar to the proof of the previous one.

3) This follows from the fact that these spectra are the union of the connected set $\varphi(G)$ and possibly some of its holes.
\end{proof}

\begin{corollary}(\cite{SbMath})\label{cr:Fredholm_Sp}
Let $\varphi \in C(G)$.  The Fredholm essential spectrum  $\sigma_{e}(T_\varphi)$  is obtained from the set   $\varphi(G)$  by
joining the holes $\Lambda$  (if they exist) such that $\varphi-\lambda \not\in \Phi(G)$ for some  (for all) $\lambda\in\Lambda$.
\end{corollary}

\begin{proof}Similar to formula \eqref{sigma_le} one has
\begin{align}\label{sigma_re}
\sigma_{re}(T_\varphi)=\varphi(G)\cup\{\lambda\in\mathbb{C}\setminus\varphi(G):\alpha_{\varphi}(\lambda)\in X_-\setminus X^i\}.
\end{align}
Therefore
\begin{align*}
\sigma_{e}(T_\varphi)&=\sigma_{le}(T_\varphi)\cap \sigma_{re}
(T_\varphi)\\
&=\varphi(G)\cup\{\lambda\in\mathbb{C}\setminus\varphi(G):\alpha_{\varphi}(\lambda)\not\in X^i\}.
\end{align*}
By the reasons mentioned in the proof of the previous theorem the right-hand side here is obtained from the set   $\varphi(G)$  by
joining the holes $\Lambda$  (if they exist) such that $\varphi-\lambda \not\in \Phi(G)$ for some  (for all) $\lambda\in\Lambda$.
\end{proof}

\section{Fredholm Properties and  Index of Discrete Wiener--Hopf Operators}

As noted above, the following concepts are due to V.~ Adukov \cite{Adukov}.

\begin{definition}\label{Definition  4} Let $k\in \ell^2(X)$. {\it The Wiener--Hopf operator}
$W_k$ on $\ell^2(X_+)$ is defined by the equality
$$
W_k g:=1_{X_+}(k\ast g)
$$
where $*$ denotes the convolution of functions on $X$.

Thus,
$$
W_kg(\chi)=\sum\limits_{\xi\in X_+}k(\chi\xi^{-1}) g(\xi)\ (\chi\in X_+).
$$
\end{definition}

\begin{definition}\label{Definition  5} Let $k\in \ell^2(X)$.
We call the inverse Fourier transform of a function $k$, i.e., the function
$$
\check k(x):=\sum\limits_{\xi\in X}k(\xi)\xi(x)
$$
from $L^2(G)$, {\it the  symbol of the operator} $W_k$.
\end{definition}

\begin{theorem}\label{Proposition5} The operator $T_\varphi$ in $H^2(G)$ with symbol $\varphi\in L^2(G)$ is unitary equivalent to the operator $W_k$ in $\ell^2(X)$, where $k=\widehat \varphi$, and vice versa, for any $k\in  \ell^2(X)$, the operator $W_k$  in $\ell^2(X)$ is unitary equivalent to the operator $T_\varphi$ in $H^2(G)$ with symbol $\varphi=\check k$.
\end{theorem}

\begin{proof} The assertion of the proposition follows from the fact that for any analytic polynomial $q\in {\rm Pol}_+(G)$ and any function $\varphi\in L^2(G)$, the equality
\begin{align}\label{equiv}
T_\varphi q=\mathcal{F}^{-1}W_{\widehat \varphi}\mathcal{F}q
\end{align}
is valid.
For the proof, note that for any function $g\in  \ell^2(X)$, the equality
\begin{align}\label{ast}
\varphi \mathcal{F}^{-1} g= \mathcal{F}^{-1}(\widehat{\varphi}\ast g)
\end{align}
 holds.

 Furthermore, one has
\begin{align}\label{FPF_1}
\mathcal{F}P_+\mathcal{F}^{-1}g=1_{X_+}g, \ \ \forall g\in  \ell^2(X).
\end{align}
Indeed, if we put $q=\mathcal{F}^{-1}g$ then \eqref{FPF_1} is equivalent to 
\begin{align}\label{FP}
\mathcal{F}P_+q=1_{X_+}\mathcal{F}q, \ \ \forall q\in  L^2(G)
\end{align}
by Plancherel's theorem.
The last equality is true because if we consider the Fourier expansion
$$
q=\sum_{\chi\in X}c_\chi \chi
$$ 
then both sides of the equality \eqref{FP} are equal to $\sum_{\chi\in {X_+}}c_\chi \chi$.

Therefore, for $g\in  \ell^2(X)$ we have
$$
\mathcal{F}T_\varphi \mathcal{F}^{-1}g=\mathcal{F}P_+(\varphi \mathcal{F}^{-1} g)=\mathcal{F}P_+\mathcal{F}^{-1}(\widehat{\varphi}\ast g)=W_{\widehat \varphi}g.
$$
To complete the proof, it remains to set $q=\mathcal{F}^{-1}g$ in the last equality.
and note that $q$ runs over ${\rm Pol}_+(G)$ when $g$ runs over the space $ \ell_{00}(X_+)$ of finitely supported functions on $X_+$).
\end{proof}

Theorem \ref{Proposition5} allows us to transfer the  properties of Toeplitz operators obtained above to the case of Wiener--Hopf operators.

\begin{corollary}\label{Corollary6} The operator $W_k$ is bounded if $k\in  \ell^2(X), \check k\in L^\infty(G)$. Moreover, $\|W_k\|=\|\check k\|_\infty$.
\end{corollary}

\begin{proof}
This follows immediately from the fact that the operator $T_\varphi$ is bounded if $\varphi\in L^\infty(G)$, and also from the equality $\|T_\varphi\|=\|\varphi\|_\infty$ \cite[Theorem 1.]{SbMath}.
\end{proof}

\begin{corollary}\label{Corollary7} For any function $\varphi\in L^2(G)$ satisfying the condition $\widehat{\varphi}\in  \ell^1(X)$, the operator $T_\varphi$ is bounded.
\end{corollary}

\begin{proof}
Indeed, under the conditions of the corollary, the operator $W_{\widehat \varphi}$ is bounded by the inequality 
$$\|\widehat \varphi\ast f\|_2\leq \|\widehat \varphi\|_1\|f\|_2.$$
\end{proof}

\begin{corollary}\label{Corollary8} Let $ \check k \in C(G)$. If the operator $W_k$ is semi-Fredholm, then $ \check k$ is invertible in the algebra $L^\infty(G)$.
\end{corollary}

This follows directly from Theorem \ref{Proposition5} and \cite[Theorem 3]{SbMath}.

The following theorem gives a criteria for the Fredholm and semi-Fredholm properties and a formula for the index of the Wiener--Hopf operator in the case of a continuous symbol. The latter serves as an analog of the Gohberg--Krein formula for the index of the Toeplitz operator.

\begin{theorem}\label{Theorem 7_WH} Let $k\in \ell^2(X_+),
\check k\in C(G)$.

(i) The Wiener--Hopf operator $W_k$ on $\ell^2(X_+)$ is  Fredholm  if and only if $\check k\in \Phi(G)$. 

(ii) The Wiener--Hopf operator $W_k$ on $\ell^2(X_+)$  semi-Fredholm  if and only if $\check k\in C(G)^{-1}$.

In both cases,
$$
{\rm Ind}W_k=-{\rm ind}\check k.
$$
\end{theorem} 

This follows directly from Theorem \ref{Proposition5} and Theorems \ref{th:Fredholm}  and  \ref{th:Semi-Fredholm}.

\begin{remark}\label{rm:1} {\rm
Note that our conditions $k\in \ell^2(X_+),\check k \in C(G)$ is less restrictive than $k\in  \ell^1(X_+)$  from \cite{Adukov}. Indeed, let $X=\mathbb{Z}$ and  $X_+=\mathbb{Z}_+$. Then $G=\mathbb{T}$  the circle group. There is a power series 
$\sum_{n\in \mathbb{Z}_+} k(n) z^n$ that converges uniformly on the closed unit disc  $\overline{\mathbb{D}}$ to some  function $\check k\in C(\mathbb{T})$, but the sequence $ k(n)$ belongs to $\ell^2(\mathbb{Z}_+)$ but not to $ \ell^1(\mathbb{Z}_+)$.

The classic example of a function from the disk algebra $A(\mathbb{D})$ whose Taylor coefficients do not belong to the space $\ell^1$ is the Hardy–Littlewood series 
$$
\sum_{n=1}^{\infty} \frac{e^{i n\log n}}{n}z^n
$$
(see, for example, \cite[Chapter V, Theorem 4.2]{Zigmund}).
}
\end{remark}

\section{Spectral Properties of Wiener--Hopf Operators}

The following theorem describes the spectra and essential spectra of Wiener--Hopf operators with continuous symbols.

\begin{theorem}\label{Theorem 8_WH} Let $\check k\in C(G)$. Then

1) the essential Fredholm spectrum $\sigma_e(W_k)$
is obtained from the set $\check k(G)$
by adding those holes $\Lambda$ (if they exist) for which $\check k-\lambda\notin \Phi(G)$ for all (for one)
$\lambda\in \Lambda$;

2) the spectrum $\sigma(W_k)$ is obtained from the set
$\sigma_e(W_k)$ by adding those holes $\Lambda$ of the set $\check k(G)$
(if they exist) for
which $\check k-\lambda\in \Phi(G)\setminus \exp(C(G))$ for all
(for one) $\lambda\in \Lambda$;

{\rm 3)} the essential Weyl spectrum coincides with the spectrum:
$\sigma_{w}(W_k)=\sigma(W_k)$;

{\rm 4)} the spectra $\sigma_e(W_k)$ and $\sigma(W_k)$
are connected.
\end{theorem} 
\begin{proof} Since $\check k\in C(G)$, $k\in \ell^2(X)$ by the Planscherel theorem. Now by Theorem  \ref{Proposition5} the statement follows from the similar results for Toeplitz  operators \cite[Theorem 5]{SbMath}.
\end{proof}

\begin{corollary}\label{cr:noncompact} Let $\check k\in C(G)$. Then $W_k$ is compact if and only if it is zero.
\end{corollary}
\begin{proof} Indeed, if $W_k$ is compact, $\sigma_e(W_k)=\{0\}$ and thus $\check k(G)=\{0\}$, too.
\end{proof}

\begin{corollary}\label{cr:realsymbol} Let the symbol $\check k\in C(G)$ be real. Then
$$
\sigma_e(W_k)=\sigma(W_k)=[\min\check k, \max\check k].
$$
\end{corollary}
\begin{proof} Indeed, in this case $\check k(G)=[\min\check k, \max\check k]$ is the set without holls. 
\end{proof}

\begin{corollary}\label{Corollary6} The operator $W_k$ is bounded if $k\in  \ell^2(X), \check k\in L^\infty(G)$. Moreover, $\|W_k\|=\|\check k\|_\infty$.
\end{corollary}

\begin{proof}
This follows immediately from the fact that the operator $T_\varphi$ is bounded if $\varphi\in L^\infty(G)$, and also from the equality $\|T_\varphi\|=\|\varphi\|_\infty$.
\end{proof}

\begin{corollary}\label{Corollary7} For any function $\varphi\in L^2(G)$ satisfying the condition $\widehat{\varphi}\in  \ell^1(X)$, the operator $T_\varphi$ is bounded.
\end{corollary}

\begin{proof}
Indeed, under the conditions of the corollary, the operator $W_{\widehat \varphi}$ is bounded by the inequality 
$$\|\widehat \varphi\ast f\|_2\leq \|\widehat \varphi\|_1\|f\|_2.$$
\end{proof}

\begin{corollary}\label{Corollary8} Let $ \check k \in C(G)$. If the operator $W_k$ is semi-Fredholm, then $ \check k$ is invertible in $L^\infty(G)$.
\end{corollary}

This follows directly from Theorem \ref{Proposition5} and \cite[Theorem 3]{SbMath}.$\Box$


 Theorem \ref{Proposition5} allows also to extend
the  results of Section 4 in \cite{SbMath}  to Wiener--Hopf operators. First of all, we state the following corollary of this  Theorem.

\begin{corollary}\label{Corollary100} Let $ \check k \in L^\infty(G)$ and $R(\check k)$ denotes the essential range of $\check k$. Then
$$
R(\check k)\subseteq \sigma(W_k)\subseteq \overline{\rm conv}(R(\check k))
$$
where $\overline{\rm conv}$ stands for the closed convex hull of a set in the complex plane.
\end{corollary}

Indeed, this follows from  a corresponding result of Murphy for Toeplitz operators  \cite{Murphy} and  Theorem \ref{Proposition5}.

\

Hereinafter, by {\it holes} of a compact connected set
$V\subset \mathbb{C}$ we mean the bounded components of its complement $\mathbb{C}\setminus V$. The following theorem describes the spectra and essential spectra of Wiener--Hopf operators.

\begin{theorem}\label{Theorem9} Let $\check k\in C(G)$. Then

{\rm 1)} the essential Fredholm spectrum $\sigma_e(W_k)$
is obtained from the set $ \check k(G)$
by adding those of its holes $\Lambda$ (if they exist) for
which $ \check k-\lambda\notin \Phi(G)$ for all (for one)
$\lambda\in \Lambda$;

{\rm 2)} the spectrum $\sigma(W_k)$ is obtained from the set
$\sigma_e(W_k)$ by adding those holes $\Lambda$ of the set $ \check k(G)$
(if they exist) for
such that $ \check k-\lambda\in \Phi(G)\setminus \exp(C(G))$ for all
(for one) $\lambda\in \Lambda$;

{\rm 3)} the essential Weyl spectrum
$\sigma_{w}(W_k)=\sigma(W_k)$;

{\rm 4)} the spectra $\sigma_e(W_k)$ and $\sigma(W_k)$
are connected.
\end{theorem}

\begin{proof}
Theorem \ref{Theorem9} follows from  Theorem \ref{Proposition5} and the corresponding results for Toeplitz operators on $H^p(G)$ \cite[Theorem 5]{SbMath}.
\end{proof}

\begin{theorem}\label{th:SemiFredholm_WH} Let $\check k\in C(G)$. Then the following statements hold:

1) the  left semi-Fredholm spectrum $\sigma_{le}(W_k)$ is obtained from the set $\check k(G)$ by
joining the holes $\Lambda$ (if they exist) such that $\chi_{\check k-\lambda}\in X_-\setminus X^i$ for some (for all) $\lambda\in \Lambda$;

2) the  right  semi-Fredholm spectrum $\sigma_{re}(W_k)$ is obtained from the set $\check k(G)$ by
joining the holes $\Lambda$ (if they exist) such that $\chi_{\check k-\lambda}\in X_+\setminus X^i$ for some (for all) $\lambda\in \Lambda$.

3) Spectra  $\sigma_{le}(W_k)$ and  $\sigma_{re}(W_k)$ are connected.
\end{theorem}

\begin{proof}
Theorem \ref{th:SemiFredholm_WH} follows from  Theorem \ref{Proposition5} and Theorem \ref{th:SpectrasemiFredholm} .
\end{proof}

\section{ Some results on invertibility of Wiener--Hopf operators}

\begin{theorem}\label{th1} Let $\check k\in H^\infty(G)$.

(i) Operator $W_k$ is invertible in $ \ell^2(X)$ if and only if $\check k$ is invertible in $H^\infty(G)$.

(ii) The spectrum $\sigma(W_k)$ equals to  the spectrum $\sigma_{H^\infty}(\check k)$ of $\check k$ in the algebra $H^\infty(G)$.
\end{theorem}

\begin{proof}
(i) Operator $W_k$ is invertible in $ \ell^2(X_+)$ if and only if $T_{\check k}$ is invertible in $H^\infty(G)$ by Theorem \ref{Proposition5}. 

If $\check k$ is invertible in $H^\infty(G)$ then $\check k\psi=1$ for some $\psi\in H^\infty(G)$. But
$$
T_{\check k}T_{\psi}=T_{\check k\psi}=T_{\psi}T_{\check k}
$$
by \cite[Proposition 1]{SbMath} and $T_{\check k\psi}=I$ the unit operator in $H^2(G)$. Thus, $W_k$ is invertible in $ \ell^2(X_+)$.

Conversely, if $W_k$ is invertible in $ \ell^2(X_+)$, i.~e. 
$T_{\check k}$ is invertible in $H^2(G)$, then 
$\check k H^2(G)=T_{\check k}H^2(G)=H^2(G)$, since $T_{\check k}f=\check k f$ for all $f\in H^2(G)$.
It follows that $\check k\psi=1$ for some $\psi\in H^\infty(G)$. Then $H^2(G)=\psi\check k H^2(G)=\psi H^2(G)$. Then $\psi\in H^\infty(G)$ by \cite[Lemma 2]{SbMath}, and hence  $\check k$ is invertible in $H^\infty(G)$.

(ii) This follows from (i) and the equality 
\begin{align*}
(W_ k-\lambda I)g&= 1_{X_+}(k\ast g)-\lambda g\\
&=1_{X_+}((k-\lambda)\ast g)=W_ {k-\lambda} g
\end{align*}
which is valid for all  $g\in  \ell^2(X_+)$.
\end{proof}

\begin{lemma}\label{mod k=1} Let $\check k\in L^\infty(G)$, $|\check k|\equiv 1$, and $W_k$ is invertible in $ \ell^2(X_+)$. If $\psi\in H^\infty(G)$ satisfies $\|\check k-\overline\psi\|_{L^\infty}<1$
then $\psi$ is invertible in $H^\infty(G)$.
\end{lemma}

\begin{proof} Theorem \cite[Theorem 1]{SbMath} yields
\begin{align}\label{eq1}
\|I-T_{\check k\psi}\|_{H^2}&=\|T_{1-\check k\psi}\|_{H^2}\\
&\le \|1-\check k\psi\|_{L^\infty}=\|\overline{\check k}-\psi\|_{L^\infty}=\|\check k-\overline\psi\|_{L^\infty}.\nonumber
\end{align}
Since $\|\check k-\overline\psi\|_{L^\infty}<1$, it follows that the operator $T_{\check k\psi}=T_{\check k}T_{\psi}$ is invertible. Since $T_{\check k}$ is also invertible, operator $T_{\psi}$ is invertible, as well.
Now Theorem \ref{th1}   implies that $\psi$ is invertible in $H^\infty(G)$.
\end{proof}

\begin{theorem}\label{th6} (cf. \cite[Theorem 3.1]{Adukov}) Let $\check k\in L^\infty(G)$, $|\check k|\equiv 1$.

(i) Operator $W_k$ is left (right) invertible in $ \ell^2(X_+)$ if and only if the inequality ${\rm dist}_{L^\infty}(\check k,H^\infty(G))<1$ (respectively, ${\rm dist}_{L^\infty}(\overline{\check k},H^\infty(G))<1$) holds.

(ii)  Operator $W_k$ is  invertible in $ \ell^2(X_+)$ if and only if $\|\check k-\psi\|_{L^\infty}<1$ for some invertible $\psi$ in $H^\infty(G)$.
\end{theorem}

\begin{proof}

(i) Consider the Hankel operator $H_{\check k}:H^2(G)\to L^2(G)\ominus H^2(G)$ with symbol $\check k$ (see \cite{StPeter}). Then $H_{\check k}f+T_{\check k}f= {\check k}f$, and therefore $\|H_{\check k}f\|_{L^2}^2+\|T_{\check k}f\|_{L^2}^2= \|f\|_{L^2}^2$. If $W_k$ is left  invertible then $T_{\check k}$ is left  invertible too by  Theorem \ref{Proposition5}. Then $\|T_{\check k}f\|_{L^2}^2\ge C\|f\|_{L^2}^2$ for some  constant $C>0$. So, we conclude that $\|H_{\check k}f\|_{L^2}<1$. It remains to note that $\|H_{\check k}f\|_{L^2}={\rm dist}_{L^\infty}(\check k,H^\infty(G))$ by \cite[Theorem 1.5]{StPeter}. 

Now the statement for the right inverse follows from  Theorem \ref{Proposition5}, and the following equality for the adjoint operator $T_{\check k}^*=T_{\overline{{\check k}}}$. 

(ii) If $W_k$ is  invertible in $ \ell^2(X_+)$ then (i) yields that $\|\overline{\check k}-\psi\|_{L^\infty}<1$ for some $\psi\in H^\infty(G)$. Then $\psi$ is  invertible in $H^\infty(G)$ by Lemma \ref{mod k=1}.

Now let $\|\check k-\psi\|_{L^\infty}<1$ for some invertible $\psi$ in $H^\infty(G)$. Note that the inequality similar to \eqref{eq1} is valid if we replace $\check k$ by $\overline{\check k}$.  Thus
$\|I-T_{\overline{\check k}\psi}\le \|\check k-\psi\|_{L^\infty}<1$. It follows that the operator $T_{\overline{\check k}\psi}=T_{\check k}^*T_{\psi}$ is invertible. The operator $T_{\psi}$ is invertible, too. Indeed, since $\psi\phi=1$ for some $\phi\in H^\infty(G)$, one has
$T_\phi T_\psi=T_\psi T_\phi=T_1=I$.
We conclude that the operator $T_{\check k}^*$ is invertible, and therefore $T_{\check k}$ and $W_k$ are invertible, as well.
\end{proof}

\end{document}